\documentclass[12pt]{amsart}
\usepackage{amsmath,amssymb}
\allowdisplaybreaks

\newtheorem{thm}{Theorem}
\newtheorem{lem}[thm]{Lemma}

\newtheorem*{thm*}{Theorem}
\newtheorem*{lem*}{Lemma}
\theoremstyle{definition}
\newtheorem*{skcu}{Strong Kirchberg Conjecture (II)}
\newtheorem*{skcp}{Strong Kirchberg Conjecture (I)}
\newcommand{\IB}{{\mathbb B}}
\newcommand{\IC}{{\mathbb C}}
\newcommand{\IF}{{\mathbb F}}
\newcommand{\IM}{{\mathbb M}}
\newcommand{\IN}{{\mathbb N}}
\newcommand{\IR}{{\mathbb R}}
\newcommand{\IZ}{{\mathbb Z}}
\newcommand{\cH}{{\mathcal H}}

\newcommand{\cP}{{\mathcal P}}
\newcommand{\cQ}{{\mathcal Q}}
\newcommand{\cU}{{\mathcal U}}
\newcommand{\fC}{{\mathfrak C}}
\newcommand{\fF}{{\mathfrak F}_m^d}
\newcommand{\fK}{{\mathfrak K}}
\newcommand{\fM}{{\mathfrak M}}

\newcommand{\fQ}{{\mathfrak Q}}

\newcommand{\e}{\varepsilon}
\newcommand{\p}{\varphi}

\DeclareMathOperator*{\Lim}{Lim}
\DeclareMathOperator{\lh}{span}

\newcommand{\ip}[1]{\mathopen{\langle}#1\mathclose{\rangle}}
\newcommand{\mb}[1]{\left[ #1 \right]}

\title{Tsirelson's Problem and Asymptotically Commuting Unitary Matrices}
\author{Narutaka Ozawa}
\address{RIMS, Kyoto University, \mbox{606-8502} Kyoto, Japan}
\email{narutaka@kurims.kyoto-u.ac.jp}
\thanks{Partially supported by JSPS (23540233) and by
the Danish National Research Foundation (DNRF) through the
Centre for Symmetry and Deformation}

\date{16 February 2012}

\begin{document}
\begin{abstract}
In this note, we consider quantum correlations of bipartite systems
having a slight interaction, and reinterpret Tsirelson's problem
(and hence Kirchberg's and Connes's conjectures) in terms of
finite-dimensional asymptotically commuting positive operator valued measures.
We also consider the systems of asymptotically commuting unitary matrices and
formulate the Stronger Kirchberg Conjecture.
\end{abstract}
\maketitle
\section{Introduction}
A POVM (positive operator valued measure) with $m$ outputs is
an $m$-tuple $(A_i)_{i=1}^m$ of positive semi-definite operators
on a Hilbert space $\cH$ such that $\sum A_i=1$.
We write the convex sets of quantum correlation matrices
of two independent systems of $d$ POVMs with $m$ outputs by
\[
\cQ_c=\left\{ \mb{\ip{ \xi, A_i^k B_j^l \xi}}_{\begin{subarray}{c}k,l\\ i,j\end{subarray}} :
\begin{array}{c}
\mbox{$\cH$ a Hilbert space, $\xi\in\cH$ a unit vector}\\
\mbox{$(A_i^k)_{i=1}^m$, $k=1,\ldots,d$, POVMs on $\cH$,}\\
\mbox{$(B_j^l)_{j=1}^m$, $l=1,\ldots,d$, POVMs on $\cH$,}\\
\mbox{$[A_i^k,B_j^l]=0$ for all $i,j$ and $k,l$}
\end{array}
\right\}
\]
and
\[
\cQ_s=\mathrm{closure}\left\{ \mb{\ip{ \xi, A_i^k B_j^l \xi}}_{\begin{subarray}{c}k,l\\ i,j\end{subarray}} :
\begin{array}{c}
\mbox{$\dim\cH<+\infty$, $\xi\in\cH$ a unit vector}\\
\mbox{$(A_i^k)_{i=1}^m$, $k=1,\ldots,d$, POVMs on $\cH$,}\\
\mbox{$(B_j^l)_{j=1}^m$, $l=1,\ldots,d$, POVMs on $\cH$,}\\
\mbox{$[A_i^k,B_j^l]=0$ for all $i,j$ and $k,l$}
\end{array}
\right\}.
\]
Here $i,j,k,l$ are indices and $A_i^k$ does not mean the $k$-th power of $A_i$.
The sets $\cQ_c$ and $\cQ_s$ are closed convex subsets of $\IM_{md}(\IR_{\geq0})$
such that $\cQ_s\subset\cQ_c$.
Whether they coincide (for some/all $m,d\geq2$, $(m,d)\neq(2,2)$) is the well-known
Tsirelson problem, and the matricial version of it is known to be
equivalent to Kirchberg's and Connes's conjectures.
We refer the reader to \cite{fritz,jungeetal,spc,tsirelson} for
the literature and the proof of the equivalence.
The matricial version of Tsirelson's problem asks whether $\cQ_c^n=\cQ_s^n$ for all $n$,
where $\cQ_c^n$ and $\cQ_s^n$ are defined as follows:
\[
\cQ_c^n=\left\{ \mb{V^*A_i^k B_j^lV}_{\begin{subarray}{c}k,l\\ i,j\end{subarray}} :
\begin{array}{c}
\mbox{$\cH$ a Hilbert space, $V\colon\ell_2^n\to\cH$ an isometry}\\
\mbox{$(A_i^k)_{i=1}^m$, $k=1,\ldots,d$, POVMs on $\cH$,}\\
\mbox{$(B_j^l)_{j=1}^m$, $l=1,\ldots,d$, POVMs on $\cH$,}\\
\mbox{$[A_i^k,B_j^l]=0$ for all $i,j$ and $k,l$}
\end{array}
\right\}
\]
and
\[
\cQ_s^n=\mathrm{closure}\left\{ \mb{ V^*A_i^k B_j^lV}_{\begin{subarray}{c}k,l\\ i,j\end{subarray}} :
\begin{array}{c}
\mbox{$\dim\cH<+\infty$, $V\colon\ell_2^n\to\cH$ an isometry}\\
\mbox{$(A_i^k)_{i=1}^m$, $k=1,\ldots,d$, POVMs on $\cH$,}\\
\mbox{$(B_j^l)_{j=1}^m$, $l=1,\ldots,d$, POVMs on $\cH$,}\\
\mbox{$[A_i^k,B_j^l]=0$ for all $i,j$ and $k,l$}
\end{array}
\right\}.
\]

In this note, we consider ``slightly interacting" systems.
Suppose Alice and Bob conduct measurements by systems of operators $(A_i^{1/2})_{i=1}^m$
and $(B_j^{1/2})_{j=1}^m$ respectively. If Bob conducts a measurement immediately after
Alice's measurement of a state $\xi$, then the probability of the output $(i,j)$ is
$\|B_j^{1/2}A_i^{1/2}\xi\|^2$---and vice versa. Therefore, when they conduct measurements
of a state $\xi$ at the same time, the probability of the output $(i,j)$ is given by
$\ip{\xi,(A_i\bullet B_j)\xi}$, where $A\bullet B=(A^{1/2}BA^{1/2}+B^{1/2}AB^{1/2})/2$.
Thus, for $\e>0$, we define the quantum correlation matrices of slightly interacting systems to be
\[
\cQ_\e^n=\mathrm{closure}\left\{ \mb{V^* (A_i^k \bullet B_j^l) V}_{\begin{subarray}{c}k,l\\ i,j\end{subarray}} :
\begin{array}{c}
\mbox{$\dim\cH<+\infty$, $V\colon\ell_2^n\to\cH$ an isometry}\\
\mbox{$(A_i^k)_{i=1}^m$, $k=1,\ldots,d$, POVMs on $\cH$,}\\
\mbox{$(B_j^l)_{j=1}^m$, $l=1,\ldots,d$, POVMs on $\cH$,}\\
\mbox{$\|[A_i^k,B_j^l]\|\le\e$ for all $i,j$ and $k,l$}
\end{array}
\right\},
\]
where $\|[A,B]\|$ denotes the operator norm of the commutator $[A,B]=AB-BA$.
We note that $\cQ_\e^n$ is a closed convex subset of $\IM_{md}(\IM_n(\IC)_+)$.
Recall that a POVM $(A_i)_{i=1}^m$ is said to be \emph{projective} if
all $A_i$'s are orthogonal projections.
We also introduce the projective analogue of $\cQ_\e^n$:
\[
\cP_\e^n=\mathrm{closure}\left\{ \mb{V^* (P_i^k \bullet Q_j^l) V}_{\begin{subarray}{c}k,l\\ i,j\end{subarray}} :
\begin{array}{c}
\mbox{$\dim\cH<+\infty$, $V\colon\ell_2^n\to\cH$ an isometry}\\
\mbox{$(P_i^k)_{i=1}^m$ projective POVMs on $\cH$,}\\
\mbox{$(Q_j^l)_{j=1}^m$ projective POVMs on $\cH$,}\\
\mbox{$\|[P_i^k,Q_j^l]\|\le\e$ for all $i,j$ and $k,l$}
\end{array}
\right\}.
\]
We simply write $\cP_\e$ for $\cP_\e^1$.
The following is the main result of this note.
It probably suggests that $\cQ_c$ is more natural than $\cQ_s$ (cf.\ Introduction of \cite{fritz}).
\begin{thm*}
For every $m,d$, and $n$, one has $\cQ_c^n=\bigcap_{\e>0}\cQ_\e^n=\bigcap_{\e>0}\cP_\e^n$.
In particular, an affirmative answer to Tsirelson's problem is
equivalent to that $\bigcap_{\e>0}\cP_\e\subset\cQ_s$.
\end{thm*}

Hence, the matricial version of Tsirelson's problem would have an affirmative answer
if the following assertion holds for some/all $(m,d)$.
\begin{skcp}
Let $m,d\geq2$ be such that $(m,d)\neq(2,2)$.
For every $\kappa>0$, there is $\e>0$ with the following property.
If $\dim\cH<+\infty$, and $(P_i^k)_{i=1}^m$ and $(Q_j^l)_{j=1}^m$
is a pair of $d$ projective POVMs on $\cH$ such that
$\|[P_i^k,Q_j^l]\|\le\e$, then there are a finite-dimensional
Hilbert space $\tilde{\cH}$
containing $\cH$ and projective POVMs $(\tilde{P}_i^k)_{i=1}^m$
and $(\tilde{Q}_j^l)_{j=1}^m$ on $\tilde{\cH}$ such that
$\|[\tilde{P}_i^k,\tilde{Q}_j^l]\|=0$ and
$\|\Phi_{\cH}(\tilde{P}_i^k) - P_i^k\|\le\kappa$ and
$\|\Phi_{\cH}(\tilde{Q}_j^l) - Q_j^l\|\le\kappa$, where
$\Phi_{\cH}$ is the compression to $\cH$.
\end{skcp}

We will deal in Section \ref{sec:grpcstar} with
a parallel and equivalent conjecture in the unitary setting.
\section{Preliminary from $\mathrm{C}^*$-algebra theory}
As it is observed in \cite{fritz,jungeetal,tsirelson},
the study of quantum correlation matrices is essentially about
the algebraic tensor product $\fF\otimes\fF$ of the $\mathrm{C}^*$-alge\-bra
\[
\fF=\ell_\infty^m*\cdots*\ell_\infty^m,
\]
the unital full free product of $d$-copies of $\ell_\infty^m$.
We note that $\fF$ is $*$-isomorphic to the full group $\mathrm{C}^*$-alge\-bra
$\mathrm{C^*}\Gamma_{m,d}$ of the group $\Gamma_{m,d}=(\IZ/m\IZ)^{*d}$.
The condition $m,d\geq2$ and $(m,d)\neq(2,2)$ is equivalent to that
$\Gamma_{m,d}$ contains the free groups $\IF_r$.
We denote by $(e_i)_{i=1}^m$ the standard basis of minimal projections in $\ell_\infty^m$,
and by $(e_i^k)_{i=1}^m$ the $k$-th copy of
it in the free product $\fF$.
We also write $e_i^k$ for the elements $e_i^k\otimes 1$ in $\fF\otimes\fF$
and $f_j^l$ for $1\otimes e_j^l$.
Thus, the maximal tensor product $\fF\otimes_{\max}\fF$
is the universal $\mathrm{C}^*$-alge\-bra generated by
projective POVMs $(e_i^k)_{i=1}^m$ and $(f_j^l)_{j=1}^m$ under
the commutation relations $[e_i^k,f_j^l]=0$.
In passing, we note that $\mathrm{C^*}\Gamma\otimes_{\max}\mathrm{C^*}\Gamma$
is canonically $*$-isomorphic to $\mathrm{C^*}(\Gamma\times\Gamma)$ for any group $\Gamma$.
By Stinespring's dilation theorem (Theorem 1.5.3 in \cite{bo}), one has
\[
\cQ_c^n=\{\mb{\p(e_i^k f_j^l)}_{\begin{subarray}{c}k,l\\ i,j\end{subarray}} :
\p\colon\fF\otimes_{\max}\fF\to\IM_n(\IC) \mbox{ u.c.p.}\}
\subset\IM_{md}(\IM_n(\IC)_+).
\]
See \cite{fritz,jungeetal} for the proof.
Here u.c.p.\ stands for ``unital completely positive."

We recall the notion of quasi-diagonality.
We say a subset $\fC$ of $\IB(\cH)$ is \emph{quasi-diagonal} if
there is an increasing net $(P_r)$ of finite-rank orthogonal projections
on $\cH$ such that $P_r\nearrow 1$ in the strong operator topology
and $\|[C,P_r]\|\to0$ for every $C\in\fC$.
A $\mathrm{C}^*$-alge\-bra $\fC$ is said to be \emph{quasi-diagonal}
if there is a faithful $*$-repre\-sen\-ta\-tion $\pi$ of $\fC$ on
a Hilbert space $\cH$ such that $\pi(\fC)$ is a quasi-diagonal subset.
A $*$-repre\-sen\-ta\-tion $\pi\colon\fC\to\IB(\cH)$ is said to be \emph{essential} if
$\pi(\fC)$ does not contain non-zero compact operators.
The following theorem of Voiculescu is the most fundamental result on quasi-diagonal $\mathrm{C}^*$-alge\-bra{}s.
See Section 7 of \cite{bo} (Theorems 7.2.5 and 7.3.6) for the details.

\begin{thm}[Voiculescu \cite{voiculescuh}]\label{thm:voiculescu}
The following statements hold.
\begin{itemize}
\item
Let $\fC\subset\IB(\cH)$ be a faithful essential $*$-repre\-sen\-ta\-tion of
a quasi-diagonal $\mathrm{C}^*$-alge\-bra $\fC$.
Then, $\fC$ is a quasi-diagonal subset of $\IB(\cH)$.
\item
Quasi-diagonality is a homotopy invariant.
\end{itemize}
\end{thm}

The following is based on Brown's idea (\cite{brown} and Proposition 7.4.5 in \cite{bo}).

\begin{thm}\label{thm:qd}
The $\mathrm{C}^*$-alge\-bra{}s $\fF\otimes_{\max}\fF$
and $\mathrm{C}^*\IF_d \otimes_{\max} \mathrm{C}^*\IF_d$
are quasi-diagonal.
\end{thm}
\begin{proof}
We consider $\fF$ as a $\mathrm{C}^*$-subalgebra of $\fM=\IM_m(\IC)*\cdots*\IM_m(\IC)$.
Since the conditional expectation $\Phi$ from $\IM_m(\IC)$ onto $\ell_\infty^m$ extends
to a u.c.p.\ map $\tilde{\Phi}$ from $\fM$ to $\fF$ which restricts to
$\Phi$ on each free product component (\cite{boca}),
the canonical embedding $\fF\hookrightarrow\fM$ is indeed faithful and $\tilde{\Phi}$
is a conditional expectation from $\fM$ onto $\fF$.
It follows that $\fF\otimes_{\max}\fF\subset\fM\otimes_{\max}\fM$.
We will prove that the latter is quasi-diagonal.

Let $\theta\colon\fM\otimes_{\max}\fM\to\IB(\cH)$ be a faithful $*$-repre\-sen\-ta\-tion
on a separable Hilbert space $\cH$.
We omit writing $\theta$ for a while and denote by
$\fM''$ the von Neumann algebra generated by $\theta(\fM\otimes\IC1)$.
We write $\{e_{i,j}\}_{i,j=1}^m$ for the matrix units in $\IM_m(\IC)$ and
$\{e_{i,j}^k\}$ for the $k$-th copy of it in $\fM$.
We note that the matrix units $\{e_{i,j}^k\}$ is unitarily equivalent
to the first copy $\{e_{i,j}\}$ inside $\fM''$.
This is a well-known fact, but we include the proof for the reader's convenience.
Let $z\in\fM''$ be the central projection such that $z\fM''$ is finite and
$(1-z)\fM''$ is properly infinite (Theorem V.1.19 in \cite{takesaki}).
Then, the projections $ze_{1,1}$ and $ze_{1,1}^k$ are equivalent
since they have the same center valued trace $z/n$
(Corollary V.2.8 in \cite{takesaki}).
The projections $(1-z)e_{1,1}$ and $(1-z)e_{1,1}^k$ are also
equivalent, since they are properly infinite and have full central
support $1-z$ (Theorem V.1.39 in \cite{takesaki}).
Therefore, for each $k$, there is a partial isometry $w_k\in\fM''$
such that $w_k^*w_k=e_{1,1}$ and $w_kw_k^*=e_{1,1}^k$.
Now, $U_k=\sum_i e_{i,1} w_k^* e_{1,i}^k$ is a unitary element in $\fM''$
such that $U_k e_{i,j}^k U_k^* = e_{i,j}$ for all $i,j$ and $k$.
Since $\fM''$ is a von Neumann algebra, there is a norm-continuous path
$U_k(t)$ of unitary elements connecting $U_k(0)=1$ to $U_k(1)=U_k$.
It follows that the $*$-homo\-mor\-phism{}s $\pi_t\colon\fM\mapsto\fM''$,
$e_{i,j}^k \mapsto U_k(t) e_{i,j}^k U_k(t)^*$, give rise to
a homotopy from $\pi_0\colon\fM\hookrightarrow\fM''$ to $\pi_1\colon \fM\to\IM_m(\IC)\subset\fM''$.
Likewise, there is a homotopy $\rho_t\colon \fM \to \theta(\IC1\otimes\fM)''$
between the embedding $\rho_0$ of $\fM$ as the second tensor component
and $\rho_1$ which ranges in $\IM_m(\IC)$.
Thus, $\pi_t\times\rho_t\colon\fM\otimes_{\max}\fM\to\IB(\cH)$ is a homotopy
between the embedding $\theta$ and $\pi_1\times\rho_1$.
Therefore, $\fM\otimes_{\max}\fM$ is embeddable into a $\mathrm{C}^*$-alge\-bra
which is homotopic to $\IM_m(\IC)\otimes\IM_m(\IC)$.
Now quasi-diagonality of $\fM\otimes_{\max}\fM$ follows from Theorem~\ref{thm:voiculescu}.
The case for $\mathrm{C}^*\IF_d$ is similar (Proposition 7.4.5 in \cite{bo}).
\end{proof}
\section{Proof of Theorem}
We start the proof of the inclusion $\bigcap_{\e>0}\cQ_\e^n\subset\cQ_c^n$.
Take $m,d,n$ and $[X_{i,j}^{k,l}] \in \bigcap_{\e>0}\cQ_\e^n$ arbitrary.
Then, for every $r\in\IN$, there are a pair of
$d$ POVMs $(A_i^k(r))_{i=1}^m$ and $(B_j^l(r))_{j=1}^m$ on $\cH_r$
and a u.c.p.\ map $\p_r\colon \IB(\cH_r)\to\IM_n(\IC)$
such that $\|[A_i^k(r),B_j^l(r)]\|\le r^{-1}$ and
$\|\p_r(A_i^k(r)\bullet B_j^l(r)) - X_{i,j}^{k,l}\|\le r^{-1}$.
We consider the $\mathrm{C}^*$-alge\-bra{}s
\begin{align*}
\fM &=\prod_{r=1}^\infty\IB(\cH_r)
 =\{ (C(r))_{r=1}^\infty : C(r)\in\IB(\cH_r),\ \sup_r\|C(r)\|<+\infty\},\\
\fK &=\bigoplus_{r=1}^\infty\IB(\cH_r)
 =\{ (C(r))_{r=1}^\infty : C(r)\in\IB(\cH_r),\ \lim_r\|C(r)\|=0\},
\end{align*}
and
$\fQ=\fM/\fK$, with the quotient map $\pi\colon\fM\to\fQ$.
Then $A_i^k=\pi((A_i^k(r))_{r=1}^\infty)$ and $B_j^l=\pi((B_j^l(r))_{r=1}^\infty)$ are commuting POVMs
in $\fQ$. Fix an ultra-limit $\Lim$ and
consider the u.c.p.\ map $\tilde{\p}\colon \fM\to\IM_n(\IC)$ defined by
$\tilde{\p}((C(r))_{r=1}^\infty)=\Lim_r \p_r(C(r)) \in\IM_n(\IC)$.
It factors through $\fQ$ and one obtains a u.c.p.\ map $\p\colon \fQ\to\IM_n(\IC)$
such that $\tilde{\p}=\p\circ\pi$.
It follows that $\p(A_i^kB_j^l)=\p(A_i^k\bullet B_j^l)=X_{i,j}^{k,l}$, and
hence $[X_{i,j}^{k,l}] \in \cQ_c^n$.

For the inclusion $\cQ_c^n\subset\bigcap_{\e>0}\cP_\e^n$,
take $m,d,n$ and $[X_{i,j}^{k,l}] \in \cQ_c^n$ arbitrary.
Then, there is a u.c.p.\ map $\p\colon \fF\otimes_{\max}\fF\to\IM_n(\IC)$
such that $\p(e_i^k f_j^l)=X_{i,j}^{k,l}$.
By Stinespring's dilation theorem, there are a $*$-repre\-sen\-ta\-tion of
$\fF\otimes_{\max}\fF$ on a separable Hilbert space $\cH$
and an isometry $V\colon\ell_2^n\to\cH$ such that $\p(C)=V^*CV$ for
$C\in \fF\otimes_{\max}\fF$.
By inflating the $*$-repre\-sen\-ta\-tion, we may assume it is faithful and essential.
Since $\fF\otimes_{\max}\fF$ is quasi-diagonal (Theorem \ref{thm:qd}),
there is an increasing sequence $(P_r)_{r=1}^\infty$
of finite-rank orthogonal projections on $\cH$ such that
$P_r\nearrow 1$ in the strong operator topology and $\|[C,P_r]\|\to0$ for $C\in\fF\otimes_{\max}\fF$.
Thus, $P_r e_i^k P_r$ and $P_r f_j^l P_r$ are close to projections (as $r\to\infty$)
and one can find projective POVMs $(E_i^k(r))_{i=1}^m$ and $(F_j^l(r))_{j=1}^m$
on $P_r\cH$ such that $\|P_r e_i^k P_r-E_i^k(r)\|\to0$ and $\|P_r f_j^l P_r - F_j^l(r)\|\to0$.
We note that $\|P_rV-V\|\to0$.
It follows that $\|[E_i^k(r),F_j^l(r)]\|\to0$ and
\[
\lim_{r\to\infty}V^* (E_i^k(r)\bullet F_j^l(r)) V
=\lim_{r\to\infty} V^*E_i^k(r)F_j^l(r)V
=V^*e_i^kf_j^lV = X_{i,j}^{k,l}.
\]
This implies $[X_{i,j}^{k,l}]\in\bigcap_{\e>0}\cP_\e^n$.
\hspace*{\fill}\qedsymbol
\section{Asymptotically commuting unitary matrices}\label{sec:grpcstar}
Kirchberg's conjecture (\cite{kirchberg}) asserts that
$\mathrm{C}^*\IF_d \otimes_{\min} \mathrm{C}^*\IF_d =\mathrm{C}^*\IF_d \otimes_{\max} \mathrm{C}^*\IF_d$
for some/all $d\geq2$.
By Choi's theorem (Theorem 7.4.1 in \cite{bo}),
$\mathrm{C}^*\IF_d$ is residually finite dimensional (RFD) and
so is $\mathrm{C}^*\IF_d \otimes_{\min} \mathrm{C}^*\IF_d$.
Since finite-dimensional representations factor through the minimal tensor product,
Kirchberg's conjecture is equivalent to the assertion that
$\mathrm{C}^*\IF_d \otimes_{\max} \mathrm{C}^*\IF_d$ is RFD.
For the following, let $u_1,\ldots u_d$ be the standard unitary generators of $\mathrm{C}^*\IF_d$.
We also write $u_i$ for the elements $u_i\otimes 1$ in
$\mathrm{C}^*\IF_d \otimes \mathrm{C}^*\IF_d$ and $v_j$ for $1\otimes u_j$.
We denote by $\cU(\cH)$ the set of unitary operators on $\cH$.
For $\alpha\in\IM_d(\IM_n(\IC))$, we consider
\begin{align*}
\|\alpha\|_{\min} &= \|\sum_{i,j}\alpha_{i,j}\otimes u_i v_j
 \|_{\IM_n(\IC)\otimes\mathrm{C}^*\IF_d \otimes_{\min} \mathrm{C}^*\IF_d}\\
&=\sup\{\|\sum_{i,j}\alpha_{i,j}\otimes U_iV_j\| : k\in\IN,\ U_i,V_j \in \cU(\ell_2^k)\mbox{ s.t.\ } [U_i,V_j]=0\}
\intertext{and}
\|\alpha\|_{\max} &= \|\sum_{i,j}\alpha_{i,j}\otimes u_i v_j
 \|_{\IM_n(\IC)\otimes\mathrm{C}^*\IF_d \otimes_{\max} \mathrm{C}^*\IF_d}\\
&=\sup\{\|\sum_{i,j}\alpha_{i,j}\otimes U_iV_j\| : U_i,V_j \in \cU(\ell_2)\mbox{ s.t.\ } [U_i,V_j]=0\}.
\end{align*}
In the above expressions, one may assume $U_1=1$ and $V_1=1$ by replacing $U_i$ and $V_j$ with $U_1^*U_i$ and $V_jV_1^*$.
It follows that $\|\alpha\|_{\min}=\|\alpha\|_{\max}$ for $d=2$.
By Pisier's linearization trick, Kirchberg's conjecture is equivalent to the assertion that
$\|\alpha\|_{\min}=\|\alpha\|_{\max}$ holds for every $d\geq 3$ (or just $d=3$) and every $\alpha\in\IM_d(\IM_n(\IC))$.
See Section 12 of \cite{pisier}, Chapter 13 in \cite{bo}, and \cite{qwep} for the proof of this fact and more information.
The proof of the following lemma is omitted because it is almost the same as that of the main theorem.

\begin{lem}\label{lem:m}
For every $\alpha\in\IM_d(\IM_n(\IC))$, one has
\[
\|\alpha\|_{\max}
= \inf_{\e>0}\sup\{\|\sum_{i,j}\alpha_{i,j}\otimes U_iV_j\|
: k\in\IN,\ U_i,V_j \in \cU(\ell_2^k)\mbox{\normalfont\ s.t.\ } \|[U_i,V_j]\|\le\e\}.
\]
\end{lem}
We observe the following fact.
Suppose $\dim\cH<\infty$ and $U,V\in\cU(\cH)$ are such that $\|[U,V]\|\le\e$.
It is well-known that the pair $(U,V)$ need not be close to a
commuting pair of unitary matrices (\cite{voiculescuu}), but
after a dilation it is.
Indeed, this follows from amenability of $\IZ^2$.
Let $m=\lfloor 1/\sqrt{\e}\rfloor$ and $F=\{0,\ldots,m\}^2\subset\IZ^2$.
We define an isometry $W\colon\cH\to\ell_2\IZ^2\otimes\cH$ by
$W\xi=|F|^{-1/2}\sum_{x\in F}\delta_{x}\otimes \p(x)\xi$, where
$\p((p,q))=U^pV^q\in\cU(\cH)$ for $(p,q)\in F$.
Then, for the commuting unitary operators $u$ and $v$, acting on
$\ell_2\IZ^2\otimes\cH$ by shifting indices in $\IZ^2$ by $(-1,0)$ and $(0,-1)$ respectively,
one has
\begin{align*}
\|W^* u W - U\| & =\| \frac{1}{|F|}\sum_{x\in F\cap((-1,0)+F)} \p(x)^*\p(x+(1,0)) -U\|\\
&\le m\e + 1/(m+1) < 2\sqrt{\e}.
\end{align*}
Similarly, one has $\|W^* v W - V\|<2\sqrt{\e}$.
Since $\mathrm{C}^*\IZ^2$ is abelian (and RFD), one can find
a finite dimensional Hilbert space $\tilde{\cH}$ containing $\cH$ and
commuting unitary matrices $\tilde{U}$ and $\tilde{V}$ on $\tilde{\cH}$ such that
$\|\Phi_{\cH}(\tilde{U}) - U\|<2\sqrt{\e}$ and
$\|\Phi_{\cH}(\tilde{V}) - V\|<2\sqrt{\e}$,
where $\Phi_{\cH}\colon\IB(\tilde{\cH})\to\IB(\cH)$ is the compression.
We note that $\Phi_{\cH}(\tilde{U}) \approx U$ and
$\Phi_{\cH}(\tilde{V}) \approx V$ for any unitary elements imply
$\Phi_{\cH}(\tilde{U}\tilde{V})\approx UV$ (see, e.g., Theorem 18 in \cite{spc}).
Keeping these facts in mind, we formulate the Strong Kirchberg Conjecture (II).
\begin{skcu}
Let $d\geq2$. For every $\kappa>0$, there is $\e>0$ with the following property.
If $\dim\cH<+\infty$ and $U_1,\ldots,U_d,V_1\,\ldots,V_d\in\cU(\cH)$
are such that $\|[U_i,V_j]\|\le\e$, then there are a finite-dimensional
Hilbert space $\tilde{\cH}$ containing $\cH$ and
$\tilde{U}_i,\tilde{V}_j\in\cU(\tilde{\cH})$
such that $\|[\tilde{U}_i,\tilde{V}_j]\|=0$ and
$\|\Phi_{\cH}(\tilde{U}_i) - U_i\|\le\kappa$ and $\|\Phi_{\cH}(\tilde{V}_j) - V_j\|\le\kappa$.
\end{skcu}

We note that the analogous statement for $U_1,U_2,V$ is true, by the proof of the
following theorem plus the fact that $\mathrm{C}^*(\IF_2\times\IZ)$ is
RFD and has the LLP (local lifting property).
See Chapter 13 in \cite{bo} for the definition of the LLP and relevant results.
\begin{thm}\label{thm:skc}
The following conjectures are equivalent.
\begin{enumerate}
\item\label{skc1}
The Strong Kirchberg Conjecture (I) holds for some/all $(m,d)$.
\item\label{skc2}
The Strong Kirchberg Conjecture (II) holds for some/all $d$.
\item\label{skc3}
Kirchberg's conjecture holds and
$\mathrm{C}^*(\IF_d \times\IF_d)$
has the LLP for some/all $d\geq2$.
\item\label{skc4}
The algebraic tensor product $\mathrm{C}^*\IF_d \otimes \mathrm{C}^*\IF_d\otimes\IB(\ell_2)$
has unique $\mathrm{C}^*$-norm.
\end{enumerate}
\end{thm}

We note that it is not known whether $\mathrm{C}^*(\IF_d \times\IF_d)$ has the LLP,
but it is independent of $d\geq2$ and equivalent to that
the LLP is closed under the maximal tensor product.
Also it is equivalent to the LLP for $\mathrm{C^*}(\Gamma_{m,d}\times\Gamma_{m,d})$.
This problem seems to be independent of Kirchberg's conjecture.
We will only prove the equivalence $(\ref{skc2})\Leftrightarrow(\ref{skc3})$,
because the proof of $(\ref{skc1})\Leftrightarrow(\ref{skc3})$ is very similar
and $(\ref{skc3})\Leftrightarrow(\ref{skc4})$ is an immediate consequence of
the tensor product characterization of the LLP (see \cite{kirchberg} and Chapter 13 in \cite{bo}).

\begin{lem}\label{lem:llp}
The following conjectures are equivalent:
\begin{enumerate}
\item\label{asl1}
For every $\kappa>0$, there is $\e>0$ with the following property.
If $\dim\cH<+\infty$ and $U_1,\ldots,U_d,V_1\,\ldots,V_d\in\cU(\cH)$
are such that $\|[U_i,V_j]\|\le\e$, then there are a (not necessarily finite-dimensional)
Hilbert space $\tilde{\cH}$ containing $\cH$ and
$\tilde{U}_i,\tilde{V}_j\in\cU(\tilde{\cH})$
such that $\|[\tilde{U}_i,\tilde{V}_j]\|=0$ and
$\|\Phi_{\cH}(\tilde{U}_i) - U_i\|\le\kappa$ and $\|\Phi_{\cH}(\tilde{V}_j) - V_j\|\le\kappa$.
\item\label{asl2}
The $\mathrm{C}^*$-alge\-bra $\mathrm{C}^*(\IF_d \times \IF_d)$ has the LLP.
\end{enumerate}
\end{lem}
\begin{proof}
$(\ref{asl1})\Rightarrow(\ref{asl2}):$
To prove the LLP of $\mathrm{C}^*$-alge\-bra $\mathrm{C}^*(\IF_d \times \IF_d)$,
it suffices to show that the surjective $*$-homo\-mor\-phism $\pi$ from
$\mathrm{C}^*(\IF_{2d})=\mathrm{C}^*(w_1,\ldots,w_d,w'_1,\ldots,w'_d)$ onto
$\mathrm{C}^*(\IF_d \times \IF_d)$, $w_i\mapsto u_i$ and $w'_j\mapsto v_j$,
is locally liftable.
By the Effros--Haagerup theorem (Theorem C.4 in \cite{bo}),
this follows once it is shown that the canonical surjection
\[
\Theta\colon \IB(\ell_2)\otimes_{\min}\mathrm{C}^*(\IF_{2d})/\IB(\ell_2)\otimes_{\min}\ker\pi
\to \IB(\ell_2)\otimes_{\min}\mathrm{C}^*(\IF_d \times \IF_d)
\]
is isometric.
Let $u_0=1=v_0$ and $E=\lh\{u_i,v_j : 0\le i,j\le d\}$
be the operator subspace of $\mathrm{C}^*(\IF_d \times \IF_d)$.
By Pisier's linearization trick, it is enough to check that
$\Theta$ is (completely) isometric on $\IB(\ell_2) \otimes E$.
For this, take $\alpha\in\IM_{d+1}(\IB(\ell_2))$ arbitrary and
let
\[
\lambda=\| \sum \alpha_{i,j}\otimes u_iv_j\|_{
\IB(\ell_2)\otimes_{\min}\mathrm{C}^*(\IF_{2d})/\IB(\ell_2)\otimes_{\min}\ker\pi}.
\]
Let $(e_n)_{n=1}^\infty$ be a quasi-central approximate unit for $\ker\pi$ in $\mathrm{C}^*(\IF_{2d})$,
and let $w_i(n)=(1-e_n)^{1/2}w_i(1-e_n)^{1/2}+e_n$ and $w'_j(n)$ likewise
(although the proof will equally work for $w'_j(n)=w'_j$).
Then, one has
\[
\| \sum \alpha_{i,j}\otimes w_i(n)w'_j(n)\|_{\IB(\ell_2)\otimes_{\min}\mathrm{C}^*(\IF_{2d})}
\geq \lambda,
\]
\[
\lim_n\|[w_i(n),w'_j(n)]\| = \lim_n\|(1-e_n)^2[w_i,w'_j]\| = \|\pi([w_i,w'_j])\|=0,
\]
and $\lim_n\|[w_i^*(n),w'_j(n)]\| =0$.
Since $\mathrm{C}^*(\IF_{2d})$ is RFD,
one can find a finite-dimensional $*$-repre\-sen\-ta\-tion $\sigma_n$ such that
\[
\| \sum \alpha_{i,j}\otimes \sigma_n(w_i(n)w'_j(n))
\|_{\IB(\ell_2)\otimes_{\min}\sigma_n(\mathrm{C}^*(\IF_{2d}))}
\geq \lambda-\frac{1}{n}.
\]
For every contractive matrices $x$ and $y$, we consider
the unitary matrices defined by
\[
U_x=\left[\begin{smallmatrix}
x & \sqrt{1-xx^*} & &\\ \sqrt{1-x^*x} & -x^* & &\\ & & x & \sqrt{1-xx^*}\\ & & \sqrt{1-x^*x} & -x^*
\end{smallmatrix}\right]
\ \mbox{ and }\
V_y=\left[\begin{smallmatrix}
y & & \sqrt{1-yy^*} &\\ & y & & \sqrt{1-yy^*}\\ \sqrt{1-y^*y} & & -y^* &\\ & \sqrt{1-y^*y} & & -y^*
\end{smallmatrix}\right].
\]
We observe that the $(1,1)$-entry of $U_xV_y$ is $xy$, and
if $\|[x,y]\|\approx0$ and $\|[x^*,y]\|\approx0$, then
$\|[U_x,V_y]\|\approx0$. Thus, applying the assumption (\ref{asl1}) to
$U_{\sigma_n(w_i(n))}$ and $V_{\sigma_n(w'_j(n))}$,
one may find unitary operators
$\tilde{U}_i(n)$, $\tilde{V}_j(n)$ and the compression $\Phi_n$
such that $[\tilde{U}_i(n),\tilde{V}_j(n)]=0$,
$\|\Phi_n(\tilde{U}_i(n))-U_{\sigma_n(w_i(n))}\|\to0$, and
$\|\Phi_n(\tilde{V}_j(n))-V_{\sigma_n(w'_j(n))}\|\to0$.
It follows that
\begin{align*}
\| \sum \alpha_{i,j}\otimes u_iv_j\|_{\IB(\ell_2)\otimes_{\min}\mathrm{C}^*(\IF_d \times \IF_d)}
&\geq \limsup_{n\to\infty} \| \sum \alpha_{i,j}\otimes \tilde{U}_i(n)\tilde{V}_j(n)\|\\
&\geq \limsup_{n\to\infty} \| \sum \alpha_{i,j}\otimes \Phi_n(\tilde{U}_i(n)\tilde{V}_j(n))\|\\
&= \limsup_{n\to\infty} \| \sum \alpha_{i,j}\otimes U_{\sigma_n(w_i(n))}V_{\sigma_n(w'_j(n))}\|\\
&\geq \limsup_{n\to\infty} \| \sum \alpha_{i,j}\otimes \sigma_n(w_i(n)w'_j(n))\|\\
&\geq\lambda.
\end{align*}
This proves that $\Theta$ is isometric on $\IB(\ell_2) \otimes E$,
and the assertion (\ref{asl2}) follows.

$(\ref{asl2})\Rightarrow(\ref{asl1}):$
Suppose that the assertion (\ref{asl1}) does not hold for some $\kappa>0$.
Thus, there are unitary operators $U_i(n)$ and $V_j(n)$ on $\cH_n$
with $\|[U_i(n),V_j(n)]\|\to0$ which witness a violation of the conclusion of (\ref{asl1}).
We consider the $\mathrm{C}^*$-alge\-bra{}s $\fM=\prod \IB(\cH_n)$ and $\fQ=\prod \IB(\cH_n)/\bigoplus\IB(\cH_n)$,
with the quotient map $\pi\colon\fM\to\fQ$.
Then, $U_i=\pi((U_i(n))_{n=1}^\infty)$ and $V_j=\pi((V_j(n))_{n=1}^\infty)$ are
commuting systems of unitary elements in $\fQ$, and the map $u_i\mapsto U_i$, $v_j\mapsto V_j$
extends to a $*$-homo\-mor\-phism on $\mathrm{C}^*(\IF_d \times \IF_d)$.
By the assumption (\ref{asl2}), one may find a u.c.p.\ map
$\p\colon\mathrm{C}^*(\IF_d \times \IF_d)\to\fM$ such that
$\pi(\p(u_i))=U_i$ and $\pi(\p(v_j))=V_j$.
We expand $\p$ as $(\p_n)_{n=1}^\infty$ and see
$\|U_i(n) - \p_n(u_i)\|\to0$ and $\|V_j(n) - \p_n(v_j)\|\to0$.
Take $N$ such that
$\|U_i(N) - \p_N(u_i)\|<\kappa$ and $\|V_j(N) - \p_N(v_j)\|<\kappa$.
By Stinespring's dilation theorem, there are a $*$-repre\-sen\-ta\-tion
$\sigma\colon\mathrm{C}^*(\IF_d \times \IF_d)\to\IB(\tilde{\cH})$
and an isometry $W\colon \cH_N\to\tilde{\cH}$
such that $\p_N(x) = W^*\sigma(x)W$.
Thus identifying $\cH_N$ with $W\cH_N$, one obtains unitary operators
$\tilde{U}_i=\sigma(u_i)$ and $\tilde{V}_j=\sigma(v_j)$ which
satisfy the conclusion of the assertion (\ref{asl1}) for $U_i(N)$ and $V_j(N)$.
This is a contradiction to the hypothesis.
\end{proof}

The analogue of Lemma~\ref{lem:llp} also holds in the projective setting,
and it can be proven using the following dilation lemma.
\begin{lem}\label{lem:acppovm}
Let $m\in\IN$ be fixed and $(A_i(n))_{i=1}^m$ and $(B_j(n))_{j=1}^m$
be sequences of POVMs on $\cH_n$ such that $\lim_n\|[A_i(n),B_j(n)]\|=0$.
Then, there are sequences of projective POVMs $(P_i(n))_{i=1}^m$ and $(Q_j(n))_{j=1}^m$
on $\ell_2^{m+1}\otimes\ell_2^{m+1}\otimes\cH_n$ such that
$\lim_n\|[P_i(n),Q_j(n)]\|=0$ and $\Phi_n(P_i(n))=A_i(n)$, $\Phi_n(Q_j(n))=B_j(n)$,
and $\Phi_n(P_i(n)Q_j(n))=A_i(n)B_j(n)$.
Here $\Phi_n$ denotes the compression to $\IC\delta_1\otimes\IC\delta_1\otimes\cH_n\cong\cH_n$.
\end{lem}
\begin{proof}
Let $X(n)=[A_1(n)^{1/2}\ \cdots\ A_m(n)^{1/2}]\in\IM_{1,m}(\IB(\cH_n))$, and consider the unitary element
\[
U(n)=\begin{bmatrix} X(n) & 0 \\ \sqrt{1-X(n)^*X(n)} & -X(n)^* \end{bmatrix} \in \IM_{m+1}(\IB(\cH_n)).
\]
We denote by $E_i(n)$ the orthogonal projection in $\IM_{m+1}(\IB(\cH_n))$ onto the $i$-th coordinate,
and define $P_i'(n)=U(n)E_i(n)U(n)^*$ for $i=1,\ldots,m-1$ and $P_m'(n)=U(n)(E_m(n)+E_{m+1}(n))U(n)^*$.
Then, $(P_i'(n))_{i=1}^m$ is a projective POVM on $\ell_2^{m+1}\otimes\cH_n$
whose $(1,1)$-entry is $(A_i(n))_{i=1}^m$.
Similarly, one obtains a projective POVM $(Q_j'(n))_{j=1}^m$.
Define $\sigma_{p,3}\colon \IB(\ell_2^{m+1}\otimes\cH_n)\to \IB(\ell_2^{m+1}\otimes\ell_2^{m+1}\otimes\cH_n)$
by $C\otimes D\mapsto C\otimes 1\otimes D$ if $p=1$, and
$C\otimes D\mapsto 1\otimes C\otimes D$ if $p=2$;
and let $P_i(n)=\sigma_{1,3}(P_i'(n))$ and $Q_j(n)=\sigma_{2,3}(Q_j'(n))$.
Since $\lim_n\|[A_i(n),B_j(n)]\|=0$, the entries of
$P_i'(n)$ asymptotically commute with those of $Q_j'(n)$.
It follows that $\lim_n\|[P_i(n),Q_j(n)]\|=0$. They also satisfy the other conditions.
\end{proof}

We are now ready for the proof of Theorem \ref{thm:skc}.
\begin{proof}
$(\ref{skc2})\Rightarrow(\ref{skc3}):$
Assume the assertion (\ref{skc2}).
Then, Lemma~\ref{lem:m} implies that $\|\alpha\|_{\max}=\|\alpha\|_{\min}$
for every $\alpha\in\IM_{d+1}(\IM_n(\IC))$ and hence Kirchberg's conjecture follows.
Lemma~\ref{lem:llp} implies that $\mathrm{C}^*(\IF_d \times\IF_d)$
has the LLP.

$(\ref{skc3})\Rightarrow(\ref{skc2}):$
Assume the assertion (\ref{skc3}).
Then, by Lemma~\ref{lem:llp}, one has the Strong Kirchberg Conjecture (II)
for a possibly infinite-dimensional $\tilde{\cH}$.
Since Kirchberg's conjecture is assumed and
$\mathrm{C}^*(\IF_d \times\IF_d) \cong \mathrm{C}^*\IF_d\otimes_{\min} \mathrm{C}^*\IF_d$
is RFD, one can reduce $\tilde{\cH}$ to a finite-dimensional Hilbert space,
up to a perturbation. See Theorem 1.7.8 in \cite{bo}.
\end{proof}

\subsection*{Final Remarks and Acknowledgment}
The main theorem equally holds for three or more commuting systems.
Although it is stated as ``the Strong Kirchberg Conjecture,"
the author thinks that both Kirchberg's and the LLP conjectures
for $\mathrm{C}^*(\IF_d \times\IF_d)$ would have negative answers.
This research came out from the author's lectures
for ``Masterclass on sofic groups and applications to operator algebras"
(University of Copenhagen, 5--9 November 2012).
The author gratefully acknowledges the kind hospitality provided by
University of Copenhagen during his stay in Fall 2012.

\end{document}